\documentclass[11pt]{article}
\usepackage[top=1in,bottom=1.25in,left=1.25in,right=1.25in]{geometry}
\usepackage{graphicx} 
\usepackage{amsmath,amsfonts,color}
\usepackage{amsthm, amssymb}
\usepackage{cleveref}
\usepackage{authblk}

\newcommand\Log{\textrm{Log}}
\newcommand\pa{\partial}
\newcommand\half{\frac{1}{2}}
\newcommand\Rez[1]{\mathrm{Re}( #1 )}

\newcommand\bX{\boldsymbol{X}}
\newcommand\bZ{\mathbb{Z}}
\newcommand\bR{\mathbb{R}}

\newcommand\bx{\boldsymbol{x}}
\newcommand\by{\boldsymbol{y}}
\newcommand\bn{\boldsymbol{n}}
\newcommand\bl{\boldsymbol{\ell}}
\newcommand\bk{\boldsymbol{k}}
\newcommand\bu{\boldsymbol{u}}
\newcommand\bv{\boldsymbol{v}}
\newcommand\bz{\boldsymbol{z}}
\newcommand\bzero{\boldsymbol{0}}
\newcommand\bj{\boldsymbol{j}}
\newcommand\cD{\mathcal{D}}
\newcommand\cS{\mathcal{S}}

\newtheorem{proposition}{Proposition}

\newtheorem{remark}{Remark}
\newtheorem{lemma}{Lemma}

\newtheorem{theorem}{Theorem}

\title{On quadrature for singular integral operators
with complex symmetric quadratic forms}
\author[1]{Jeremy Hoskins}
\author[2]{Manas Rachh}
\author[3]{Bowei Wu}
\affil[1]{Department of Statistics, University of Chicago, Chicago, IL  60637}
\affil[2]{Center for Computational Mathematics, Flatiron Institute, New York, NY, 10010}
\affil[3]{Department of Mathematics and Statistics, University of Massachusetts Lowell, Lowell, MA 01854}
\date{December 11, 2023}

\begin{document}

\maketitle

\begin{abstract}
This paper describes a trapezoidal quadrature method for the discretization of weakly singular, singular and hypersingular boundary integral operators with complex symmetric quadratic forms. Such integral operators naturally arise when complex coordinate methods or complexified contour methods are used for the solution of time-harmonic acoustic and electromagnetic interface problems in three dimensions. The quadrature is an extension of a locally corrected punctured trapezoidal rule in parameter space wherein the correction weights are determined by fitting moments of error in the punctured trapezoidal rule, which is known analytically in terms of the Epstein zeta function. 
In this work, we analyze the analytic continuation of the Epstein zeta function and the generalized Wigner limits to complex quadratic forms; this analysis is essential to apply the fitting procedure for computing the correction weights. We illustrate the high-order convergence of this approach through several numerical examples.
\end{abstract}

\section{Introduction}

In this short paper we consider quadrature formulae for integrals of the form
\begin{align}\label{eqn:int}
I =\int_{D} \frac{g(\bv)}{\left(\bv^T A \bv\right)^s}\,{\rm d}\bv,
\end{align}
where $D$ is either $\mathbb{T} \times \mathbb{R}$ or $\mathbb{R}^2,$ and $A$ is a $2\times 2$ complex symmetric matrix. Additionally, we assume that $s \in \mathbb{C} $ (provided the integral makes sense), $g \in C^\infty(D)$ is a complex compactly-supported smooth function defined on $D.$ In practice, the regularity assumptions on $g$ can be relaxed substantially. We note that when $\Re(s)>1$, the integral is hypersingular.

Integrals of this form arise quite naturally in a number of contexts, particularly in the physical modeling of infinite interfaces, or objects, in acoustical, optical, or fluidic scattering problems. In this setting, $g$ encodes both a smooth component of the kernel and the density of a layer potential for representing the solution, and $A$ encodes the first fundamental form at a point on the surface. For purely real $A$, in \cite{wu2021corrected} and later in \cite{wu2023unified}, it was shown how to obtain suitable quadrature formulae by modifying the standard trapezoidal rule. More concretely, for $D = \mathbb{R}^2,$ if the region is discretized using equispaced points with spacing $h,$ then the classical punctured trapezoidal rule is
\begin{align}\label{eqn:punc_trap}
I \approx T_{h}:={\sum_{\bj\in \mathbb{Z}^2}} '\frac{g(\bj h)}{\left(\bj^{T} A \bj\right)^s} h^{2-2s} \,,
\end{align}
where $\bj = (j_{1},j_{2})$ and the prime in the sum denotes the fact that the $(0,0)$ term is excluded.
The singularity in the denominator means that one would expect it to converge like $O(h^{2-2s}).$ For real symmetric matrices $A$, it has been shown in \cite{wu2021corrected}
that the Epstein zeta function~\cite{epstein1903theorie,epstein1906theorie} denoted by $Z_{A}(s)$ is a {\it diagonal correction}, satisfying
$$I - T_{h} -Z_{A}(s) g(\bzero)h^{2-2s} = O(h^{4-2s})\,.$$

In this short note we extend the results in \cite{wu2021corrected} to the case of complex symmetric quadratic forms. As part of this, we analyze the analytic continuation of the Epstein zeta function to complex symmetric quadratic forms with positive definite real parts. We also extend the Wigner limits of \cite{wu2023unified} for complex quadratic forms. The result is a simple formula for a corrected trapezoidal rule which can be applied to approximate integrals of the form \cref{eqn:int}.

The remainder of the paper is organized as follows. In Section \ref{sec:main}, we state the main analytical results of the paper and present a sketch of their proofs. In Section \ref{sec:lay_pot}, we review definitions of layer potentials of acoustic and electromagnetic scattering, and discuss an application in which evaluation of integrals of the form~\cref{eqn:int} naturally arise. Finally, in Section \ref{sec:numerics} we give details of the numerical implementation of the modified trapezoidal quadrature, and show its application to solving scattering problems involving infinite interfaces, and infinitely long objects.

\section{Main results}\label{sec:main}
In this section we present the main result of this paper (\Cref{thm:diag}) which constructs the diagonal correction in the Trapezoidal rule for evaluating~\cref{eqn:int}. In the following, we suppose that $A$ is a complex symmetric matrix given by
\begin{equation}
A = \begin{bmatrix} E & F\\ F & G\end{bmatrix},
\end{equation}
where $E,F,G\in\mathbb{C}$, and both ${\Re}(A)$ and ${\Re}(A^{-1})$ are positive definite. Let $\| A \|$ denote the spectral norm of $A$. 
Furthermore, let $Q_A$ denote the quadratic form associated with $A,$
\begin{equation}
Q_A(\bv) = Q_{A}(v_{1},v_{2}) := Ev_{1}^2+2Fv_{1}v_{2}+Gv_{2}^2.
\end{equation}
Finally, let $\Gamma(\cdot)$ denote the Gamma function, $\Gamma(\cdot,\cdot)$ denote the incomplete Gamma function, and $Z_{A}(s)$ denote the Epstein zeta function given by
\begin{align}\label{eqn:zeta_def}
    Z_A(s):= {\sum_{\bj\in \mathbb{Z}^{2}}} ' Q_A(\bj)^{-s} \,.
\end{align}
Note that the sum is convergent and analytic in the region $\Re{(s)} > 1$.

We begin by considering the analytic continuation of the Epstein zeta function in the following proposition.
\begin{proposition}
\label{prop:zeta_ext}
    Suppose that $Z_A(s)$ is the Epstein zeta function defined by~\cref{eqn:zeta_def}. Suppose that $\Re(A)$ and $\Re(A^{-1})$ are positive definite. Then $Z_A$ has an analytic extension valid for all $s\in \mathbb{C}\setminus\{1\}.$
\end{proposition}
\begin{proof}
Since ${\Re}(Q_A(\bj))>0$ for all $\bj \in \mathbb{Z}^2\setminus \{\bzero\}$, we have
\begin{align}\label{eqn:zeta_exp}
\frac{\Gamma(s)}{\pi^s} Z_{A}(s) = \frac{\Gamma(s)}{\pi^s}\,\sum_{\bj \in \mathbb{Z}^{2}}\,' \frac{1}{Q_A(\bj)^s} = &\sum_{\bj \in \mathbb{Z}^{2}}\,'\,\left[\frac{\Gamma(s,\pi Q_A(\bj))}{\left(\pi Q_A(\bj)\right)^s}+\frac{1}{\sqrt{{\rm det}A}}\frac{\Gamma(1-s,\pi Q_{A^{-1}}(\bj))}{\left(\pi Q_{A^{-1}}(\bj)\right)^{1-s}}\right]\\
&+\frac{1}{(s-1)\sqrt{{\rm det}A}} -\frac{1}{s},\nonumber
\end{align}
for all $s,$ with $\Re(s) >1.$ The proof is essentially identical to that of the real case in \cite[eq.(1.2.11)]{borwein2013lattice}. For completeness, a proof of \cref{eqn:zeta_exp} is given in~\labelcref{app:split}. Clearly, the left-hand side is well-defined and analytic provided that $\Re(s) >1$, and $\Re{(A)}$
is positive definite. On the other hand, the right hand side is well-defined and analytic for all $s \in \mathbb{C} \setminus \{1\}$, when $\Re{(A)}$ and $\Re{(A^{-1})}$ are positive definite, and thus serves as an analytic extension of the left hand side over that set.
\end{proof}
\begin{remark}
    In the literature, the analytic continuation of $Z_A(s)$ in $s$ is classical. To the best of our knowledge, the analytic continuation in $A$, though perhaps unsurprising, is new, particularly for non-Hermitian matrices. 
\end{remark}
\begin{remark}
In terms of numerical algorithms, the incomplete Gamma function in \cref{eqn:zeta_exp} can be evaluated at complex arguments efficiently using Legendre's continued fraction expansion~\cite{temme94}. 
\end{remark}
\begin{remark}
One special case of particular importance is when $A$ is a diagonal matrix. In this situation, the restriction on positive definiteness of $\Re{(A)}$, and $\Re{(A^{-1})}$, can be relaxed to 
\begin{equation}
\frac{E}{|E|} + \frac{G}{|G|} \neq 0 \, .
\end{equation}
In such situations it is easy to show that there exists $\xi \in \mathbb{C},$ $|\xi|=1,$ such that $\Re (\xi A)$ and $\Re(\bar{\xi}A^{-1})$ are both positive definite. $Q_A$ is related to $Q_{\xi A}$ by $Q_A(\bj) = Q_{\xi A}(\bj)/\xi.$ Thus,  one can use the above arguments to obtain an expression for $Z_A(s)$ in terms of $Z_{\xi A}(s),$ even though the right-hand side of  \cref{eqn:zeta_exp} is undefined. 
\end{remark}

In the next result, we bound the error incurred when the series appearing on the right-hand side of \cref{eqn:zeta_exp} is truncated.
\begin{proposition}
Suppose that $\Re(A)$ and $\Re(A^{-1})$ are both positive definite.  Given $\epsilon >0,$ suppose that $\rho$ satisfies
$$ \pi\rho^2 > \max \{ \|A\| , s \|(\Re(A))^{-1}\| , \|A^{-1}\| , (1-s) \|(\Re(A^{-1}))^{-1}\|\}$$
and
\begin{equation}
\label{eq:rho_res}
\rho \ge \max\left\{\frac{1}{\|(\Re (A))^{-1}\|^{1/2}}\sqrt{\log \frac{4\pi}{\epsilon\,\|(\Re (A))^{-1}\|}}+1,\frac{1}{\|(\Re (A^{-1}))^{-1}\|^{1/2}}\sqrt{\log \frac{4\pi}{\epsilon\,\|(\Re (A^{-1}))^{-1}\|}}+1 \right\},
\end{equation}
then 
\begin{align}\label{}
\frac{\Gamma(s)}{\pi^s}\,\sum_{\bj \in \mathbb{Z}^{2}}\,' \frac{1}{Q_A(\bj)^s} = &\sum_{\substack{\bj \in \mathbb{Z}^{2} \\j_{1}^2+j_{2}^2\le \rho^2}}\,'\,\left[\frac{\Gamma(s,\pi Q_A(\bj))}{\left(\pi Q_A(\bj)\right)^s}+\frac{1}{\sqrt{{\rm det}A}}\frac{\Gamma(1-s,\pi Q_{A^{-1}}(\bj))}{\left(\pi Q_{A^{-1}}(\bj)\right)^{1-s}}\right]\\
&+\frac{1}{(s-1)\sqrt{{\rm det} A}} -\frac{1}{s} + R,\nonumber
\end{align}
where $|R| < \epsilon.$
\end{proposition}
\begin{proof}
    We begin by recalling the following asymptotic bound on the incomplete Gamma function, 
    \begin{align*}
        \left|\Gamma(s,z) - z^{s-1} e^{-z} \right| \le \frac{e^{-\Re(z)} |z^{s-2}|}{\Re (z -s) +1},
    \end{align*}
    with $s,z \in \mathbb{C}$, $\Re(z)>\Re(s)-1$, and $|{\rm arg}(z)|\le \pi/2 - \delta$ for some $\delta>0$~\cite{olver}.
Suppose that the summation in the first term on the right hand side of \cref{eqn:zeta_exp} is truncated at $j_{1}^2+j_{2}^2 \le \rho^2$, and let $R_{1}$ denote the remainder given by
    $$R_1 := \sum_{\substack{\bj \in \mathbb{Z}^{2} \\j_{1}^2+j_{2}^2>\rho^2}} \frac{\Gamma(s,\pi Q_A(\bj))}{(\pi Q_A(\bj))^s} \, .$$
Substituting the upper bound of the incomplete Gamma function into the above expression, we find that
    $$|R_1| \le \sum_{\substack{\bj \in \mathbb{Z}^{2} \\j_{1}^2+j_{2}^2>\rho^2}} \frac{e^{-\pi \Re(Q_A(\bj))}}{\pi |Q_A(\bj)|}\left[1+ \frac{1}{\pi |Q_A(\bj)|\,(\Re(\pi Q_A(\bj)-s)+1)} \right]. $$
    It follows immediately that for all $\pi \rho^2 \ge \max \{\|A^{-1}\|, s \|(\Re (A))^{-1}\| ,1 \},$
    \begin{equation}
    \nonumber
    \begin{aligned}
    |R_1| &\le 2 \sum_{\substack{\bj \in \bZ^{2} \\j_{1}^2+j_{2}^2> \rho^2}} e^{-\|(\Re (A))^{-1}\| (j_{1}^2+j_{2}^2)}  \\
    &\le 4\pi \int_{\rho-1}^\infty e^{-\|(\Re (A))^{-1}\| r^2}\,r \,{\rm d}r = \frac{2\pi}{\|(\Re (A))^{-1}\|} e^{-\|(\Re (A))^{-1}\| (\rho-1)^2} \leq \frac{\epsilon}{2} \, ,
    \end{aligned}
    \end{equation}
    where the last inequality follows from~\cref{eq:rho_res}.
    Using similar reasoning, we see that for the second series, if
        $$R_2 := \sum_{\substack{\bj \in \bZ^2 \\j_{1}^2+j_{2}^2 >\rho^2}} \frac{\Gamma(1-s,\pi Q_{A^{-1}}(\bj))}{(\pi Q_{A^{-1}}(\bj))^{1-s}}$$
        then 
        $$|R_2| \le \frac{ 2\pi}{\|(\Re A^{-1})^{-1}\|}\, e^{-\|(\Re (A^{-1}))^{-1}\| (\rho-1)^2} \leq \frac{\epsilon}{2} \, .$$
\end{proof}

Suppose that $J_{N}$ is the tensor
product set
\begin{equation}
\label{eq:tens-set}
J_{N} = \{-N, -N+1, \ldots -1,0,1,\ldots N-1,N \}^2 \, .
\end{equation} The following proposition establishes the analyticity of the difference between certain infinite sums and integrals over the plane, which are called the {\it Wigner limits} \cite{borwein2014lattice}. Its proof is a slight generalization of \cite{borwein89} which we present in~\ref{app:bor}.
\begin{proposition}
\label{prop:borwein}
Suppose $Z_A^{(N)},I_A^{(N)},$ and $W_A^{(N)}$ are given by
\begin{equation}
\label{eq:wzi-def}
\begin{aligned}
Z_A^{(N)}(s) &:= {\sum_{\bj \in J_{N}}}'Q_{A}(\bj)^{-s},\\
I_A^{(N)}(s) &:= \int_{\left[-N-\frac{1}{2}, N + \frac{1}{2}\ \right]^2} Q_{A}(\bv)^{-s}\,d\bv, \\
W_A^{(N)}(s) &:= Z_A^{(N)}(s) - I_A^{(N)}(s).
\end{aligned}
\end{equation}
Then $W_A(s):=\lim_{N\to\infty}W_A^{(N)}(s)$ is called the Wigner limit, which exists for $s \in \{0<\Re(s)< 1 \}$ and can be analytically continued to $\{\Re(s)>0\, \cap\, s\neq 1 \}$. In particular, the Wigner limit coincides with the Epstein zeta function \cref{eqn:zeta_def} when $\Re(s)>1$, that is, $W_A(s)=Z_A(s)$.
\end{proposition}

We now turn to the main result of this paper. For concreteness, we restrict our attention to the case of $D = \mathbb{R}^{2}$, the extension to $D = \mathbb{T} \times \mathbb{R}$ is straightforward. 

\begin{theorem}
\label{thm:diag}
Let $g: \mathbb{R}^2 \to \mathbb{C}$ be a smooth compactly supported function and $s\in \mathbb{R},$ $s<1.$ If $I$ denotes the integral defined in \cref{eqn:int} and $T_h$ denotes the punctured trapezoidal approximation defined in \cref{eqn:punc_trap} then
\begin{equation}
|I - T_{h} - Z_{A}(s) g(\bzero) h^{2-2s}| = O(h^{4-2s}) \,.
\label{eq:lap_diag_corr}
\end{equation}
\end{theorem}
\begin{proof}
    We split the contributions of the integral and sum into two pieces: a global piece which vanishes near the singularity; and a local piece which contains the singularity but is compactly supported. To that end, we let $\eta(\bv)$ be a radially symmetric, non-negative, compactly supported, $C^{\infty}$ function on $\bR^2$ with support in an interval $(-a,a)^2$, and $\eta(\bv) \equiv 1$ on $[-a/2,a/2]^2$ for some $a>0$. 
There exists $c_{1},c_{2}$, such that
\begin{equation}
\label{eq:tayl_g}
\begin{aligned}
g(\bv) - g(\bzero) - c_{1}v_{1} - c_{2}v_{2} &= O(|\bv|^2) \, ,
\end{aligned}
\end{equation}
holds for $\bv$ in the vicinity of $\bzero$.
We split the integral in \cref{eqn:int} into two parts
\begin{equation}
I = \int_{\bR^2} \frac{g(\bv) \eta(\bv)}{Q_A(\bv)^s} {\rm d}\bv + \int_{\bR^2} \frac{g(\bv)(1- \eta(\bv))}{Q_A(\bv)^{s}} {\rm d}\bv \, .
\end{equation}
and observe that
\begin{equation}
\label{eq:non-sing_g}
\left|\int_{\mathbb{R}^2} \frac{g(\bv)(1- \eta(\bv))}{Q_A(\bv)^s}\,{\rm d} \bv - {\sum_{\bj \in \bZ^2}}' \frac{g(\bj h)(1 - \eta(\bj h))}{Q_A(\bj h)^s} \right| = O(h^{M}) \, ,
\end{equation}
for all $M>0$. This follows from the compact support of $g$, and that $(1-\eta(\bv))$ vanishing to all orders at the origin. Hence, the integrand is a smooth and compactly supported function on all of $\bR^2$, for which the trapezoidal rule is spectrally accurate. Note that the term corresponding to $\bj = \bzero$ can be omitted in the trapezoidal approximation since $(1-\eta(\bv)) \equiv 0$ in the vicinity of $\bv = \bzero$.

We now turn our attention to the singular term in the vicinity of the origin
\begin{equation}
I_s:=\int_{\bR^2} \frac{g(\bv) \eta(\bv)}{Q_A(\bv)^{s}}{\rm d}\bv \, .\end{equation}
Clearly, it suffices to show that
\begin{equation}
I_s =  {\sum_{\bj \in \bZ^2}}' \frac{g(\bj h) \eta (\bj h))}{Q_{A}(\bj h)^{s}}h^2 - Z_{A}(s) g(\bzero) h^{2-2s} + O(h^{4-2s}).
\label{eq:I2_g}
\end{equation}

In order to establish~\cref{eq:I2_g} holds, we split the integrand once more, writing
\begin{equation}
I_s = \underbrace{\int_{\bR^2} \frac{\left(g(\bv)-g(\bzero) - c_1v_1-c_2v_2\right) \eta(\bv)}{Q_{A}(\bv)^{s}}{\rm d}\bv}_{I_{s,1}} + g(\bzero) \underbrace{\int_{\bR^2} \frac{\eta(\bv)}{Q_{A}(\bv)^{s}}{\rm d}\bv}_{I_{s,2}} + \underbrace{\int_{\bR^2} \frac{c_{1}v_{1} + c_{2}v_{2}}{Q_{A}(\bv)^{s}}\eta(\bv){\rm d}\bv}_{I_{s,3}}.
\label{eq:I2_ag}
\end{equation}
Then, using~\cref{eq:tayl_g}, we obtain
\begin{equation}
\frac{\left(g(\bv)-g(\bzero)-c_1v_1-c_2v_2\right) \eta(\bv)}{Q_{A}(\bv)^{s}} = O(|\bv|^{2-2s}) \, .
\label{eq:I2_b}
\end{equation}
The $O(|\bv|^{2-2s})$ remainder implies that the error of the trapezoidal rule when applied to $I_{s,1}$ is $O(h^{4-2s})$. 

For the integral $I_{s,2}$, let $T_h^{s,2}$ denote its punctured trapezoidal rule approximation. Then by a change of variable $\bv=\bu h$ in $I_{s,2}$,  we have 
\begin{equation}
    I_{s,2}-T_h^{s,2}=h^{2-2s}\underbrace{\left(\int_{\bR^2} \frac{\eta(\bu h)}{Q_{A}(\bu)^{s}}{\rm d}\bu - {\sum_{\bj \in \bZ^2}}' \frac{\eta (\bj h)}{Q_{A}(\bj)^{s}}\right)}_{f(h)}
\label{eq:I2_c}
\end{equation}
Note that the expression $f(h)$ above defines a $C^\infty$ function of $h$. We will show that
\begin{equation}
    f(h) = - Z_A(s) + O(h^2)~~ {\rm as}~~h\to 0,
\label{eq:I2_d}
\end{equation}
by first showing that $f(0)=-Z_A(s)$ and then $f'(0)=0$. First, to see that $f(0)=-Z_A(s)$, we use the fact that $\eta(\bzero)=1$, then
$$
\begin{aligned}
    f(0) = \lim_{h\to 0}f(h) &= \lim_{h\to 0}\left(\int_{\bR^2} \frac{\eta(\bu h)}{Q_{A}(\bu)^{s}}{\rm d}\bu - {\sum_{\bj \in \bZ^2}}' \frac{\eta (\bj h)}{Q_{A}(\bj)^{s}}\right)\\
    &= \int_{\bR^2} \frac{1}{Q_{A}(\bu)^{s}}{\rm d}\bu - {\sum_{\bj \in \bZ^2}}' \frac{1}{Q_{A}(\bj)^{s}} = -W_A(s) = -Z_A(s),
\end{aligned}
$$
where the last two equalities used Proposition \ref{prop:borwein} and Theorem \ref{thm:wigner} (\ref{app:bor}). Next, to show that $f'(0)=0$, we use the fact that $\eta(\bv)=\eta(-\bv)$ for all $\bv$, hence $f(h) = f(-h)$,
$$
f'(0) = \lim_{h\to0}\frac{f(h)-f(-h)}{2h} = \lim_{h\to0}\frac{0}{2h} = 0.
$$
This finishes the proof of \cref{eq:I2_d}, which when substituted in \cref{eq:I2_c} gives
\begin{equation}
    I_{s,2}-T_h^{s,2} = -Z_A(s)h^{2-2s} + O(h^{4-2s})
\end{equation}
For $I_{s,3}$, the odd symmetry of the integrand implies that both the integral and its trapezoidal rule approximation vanish.
This completes the verification of the estimate~\cref{eq:I2_g}. The result \cref{eq:lap_diag_corr} then follows from~\cref{eq:non-sing_g,eq:I2_g}.
\end{proof}

 \begin{remark}
    Higher-order corrections can be achieved following the same correction schemes in \cite{wu2021corrected,wu2023unified}, the key difference in this paper is that we allow the quadratic form to be complex symmetric.
\end{remark}

\section{Layer potentials and integral equations}\label{sec:lay_pot}
The computation of integrals of the form~\cref{eqn:int} arises naturally in the evaluation of layer potentials for solving integral equations on infinite surfaces when using the complexified contour approach (see~\cite{shidong2024, lu2018, bonnet2022}, for example) which is an extension of {\it perfectly matched layer} methods to boundary integral equations with infinite interfaces. In this section, we briefly review the relevant operators and equations.

 Consider the solution of the Helmholtz Dirichlet problem on a perturbed half space $\Omega$ and boundary $\pa \Omega$, see~\cref{fig:pert-halfspace}. The boundary of the region is of the form $(x_{1},x_{2},0)$ for $|\bx| > L$.  
 \begin{figure}[h]
\centering
\includegraphics[width=0.7\textwidth]{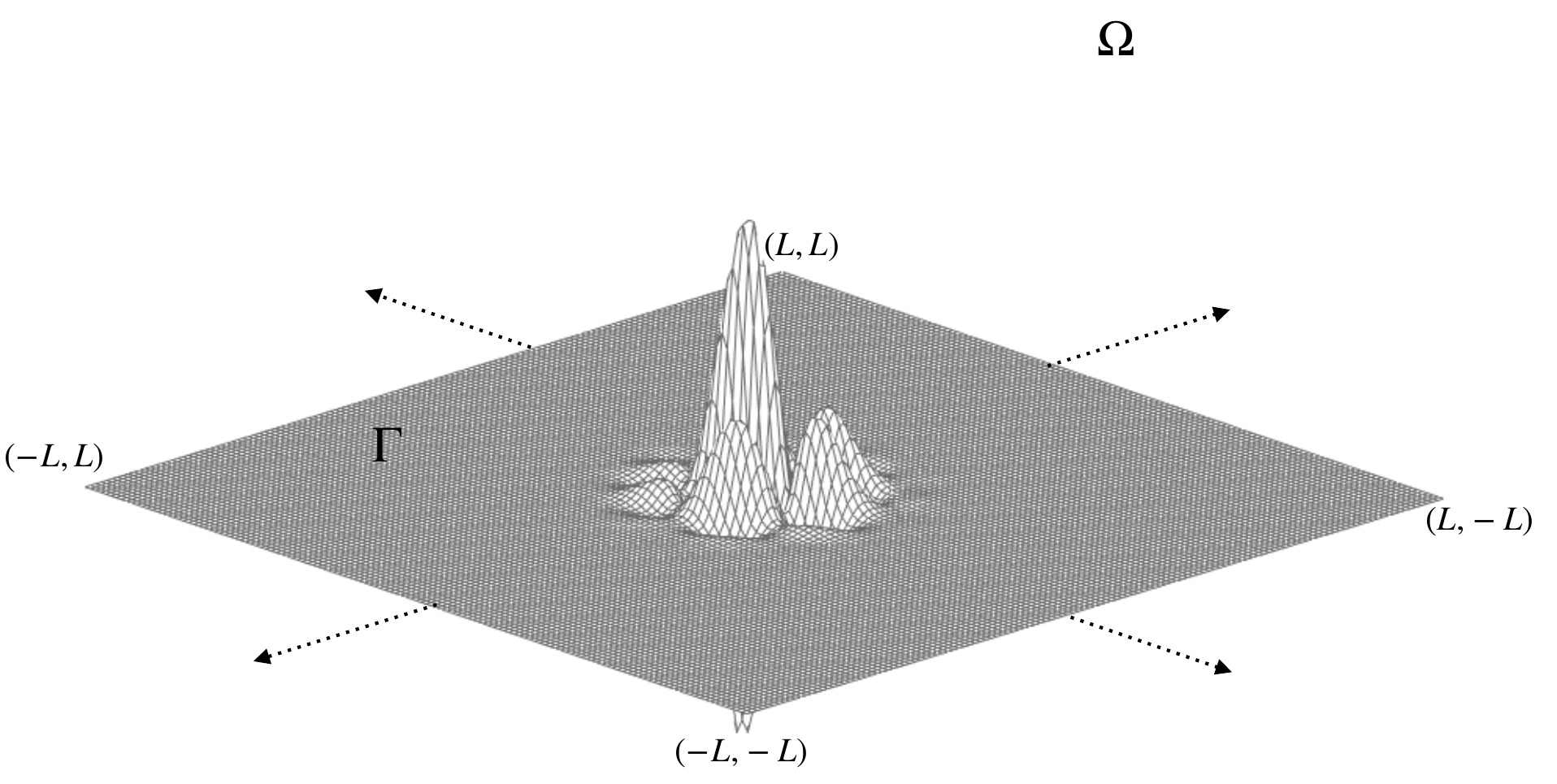}
\caption{Example geometry for Helmholtz Dirichlet scattering problem on a perturbed half-space.}
\label{fig:pert-halfspace}
\end{figure}
The potential $u$ satisfies
\begin{equation}
\begin{aligned}
(\Delta + k^2) u &= 0 \,, \quad \bx \in \Omega \,, \\
u &= f \, ,\quad \bx \in \pa \Omega \, , \\
\lim_{|\bx|\to \infty} |\bx| (\partial_{|\bx|} -ik)u  &= 0 \, , \quad \textrm{uniformly in angle.}
\end{aligned}
\end{equation} 
Here $k$ is the Helmholtz wavenumber, and $f$ is the given Dirichlet data. A standard approach to solve this boundary value problem is to represent the solution $u$ via the Helmholtz double layer potential with a density $\sigma$, denoted by $\cD[\sigma]$, and given by
\begin{equation}
u(\bx) = \cD[\sigma](\bx) = \int_{\pa \Omega} \frac{(\bx - \by) \cdot \bn(\by)}{4\pi |\bx - \by|^3} (1 - ik|\bx-\by|) \exp{(ik|\bx - \by|)} \sigma(\by) \, dS_{\by} \, ,
\end{equation}
where $\bn(\by)$ is the upward facing normal (pointing into the upper half-space) to the surface $\pa \Omega$ at $\by$. This representation satisfies the Helmholtz equation in $\Omega$, and the outgoing boundary condition when the density $\sigma$ is appropriately constrained. On imposing the boundary conditions, and using the standard jump relations for Helmholtz layer potentials~\cite{kress1989, colton2013}, the density $\sigma$ is then obtained by the solution of the following integral equation
\begin{equation}
\frac{1}{2} \sigma(\bx) + \cD^{PV}[\sigma](\bx) = f(\bx) \, , \quad \bx \in \pa \Omega \, ,
\end{equation}
where $\cD^{PV}[\sigma]$ is the restriction of the double layer to the boundary with the integral interpreted in a principal value sense.

One of the difficulties in solving this integral equation for plane wave boundary data, and incoming fields generated from volumetric sources is the slow decay of both the boundary data $f$, and the density $\sigma$. More precisely, for such physical sources, the density and the data satisfy the estimates
\begin{equation}
\sigma(\bx), f(\bx) \lesssim \frac{e^{ik|\bx|}}{|\bx|} \, ,
\end{equation}
as $|\bx| \to \infty$ (for plane wave data, the estimate holds after subtracting the contribution from Snell's law).

A recent method addresses this issue via deforming the integral equation from $(x_{1},x_{2},0) \to (x_{1} + i\psi(x_{1}), x_{2} + i\psi(x_{2}), 0)$ for $|\bx| > L$, and solving the integral equation on the deformed complex surface instead.
Let $\pa \Omega_{C}$ denote the complexified surface, given by the original surface $\pa \Omega$ for $|\bx| < L$, and by 
$(x_{1} + i\psi(x_{1}), x_{2} + i\psi(x_{2}), 0)$ otherwise.
Formally, the integral equation on the complexified interface is given by
\begin{equation}
\label{eq:complex-inteq}
\frac{1}{2}\sigma(\bx) + p.v. \int_{\pa \Omega_{C}}  \frac{(\bx - \by) \cdot \bn(\by)}{4\pi r(\bx-\by) } (1 - ik r(\bx-\by)) \exp{(ik r(\bx-\by))} \sigma(\by) \, dS_{\by} \,  = f(\bx) \,, \quad \bx \in \pa \Omega_{C} \,,
\end{equation} 
where $r(\bx)$ is as defined by
\begin{equation}
\label{eq:rdef}
r(\bx) = \sqrt{x_{1}^2 + x_{2}^2 + x_{3}^2} \,
\end{equation}
which is the analytic extension of $|\bx-\by|$ when $\bx$ or $\by$ have complex coordinates. Suppose that $\psi$ is chosen to be a monotonically increasing function of $x$, say $\psi(x) = \alpha (x-L)$ for $x>L$, and $\alpha(x+L)$ for $x<-L$, for example. In this setting, the kernel of the integral operator on the complexified surface decays exponentially, and the density and the data satisfy the estimates
\begin{equation}
|\sigma(\bx)|, |f(\bx)| \lesssim e^{-(k -\delta) \alpha \sqrt{x_{1}^2 + x_{2}^2}} \, .
\end{equation} 
For a detailed discussion on the use of complex coordinates for the solution of outgoing boundary value problems for Helmholtz equations, see~\cite{shidong2024}.

\begin{remark}
When solving the Dirichlet problem, analogous to compact obstacles, the combined field integral equation is the representation of choice, in which case the potential $u$ is represented as
\begin{equation}
\label{eq:comb-rep}
u(\bx) = \cD[\sigma](\bx) - ik \cS[\sigma](\bx) \, ,
\end{equation}
where $\cS[\sigma]$ is the single layer potential given by
\begin{equation}
\cS[\sigma](\bx) = \int_{\pa \Omega} \frac{\exp{(ik|\bx - \by|)}}{4\pi|\bx-\by|} \sigma(\by) \, dS_{\by} \, .
\end{equation}
The quadrature correction above can be easily adapted to this case.
\end{remark}

\begin{remark}
For related boundary value problems, such as interfaces with transmission boundary conditions, the integral operators involve restrictions of the Neumann data of the single and double layer potentials as well. The use of complex coordinates, and the quadratures presented in this work, for efficient discretization of the integral operators apply to this setting as well.
\end{remark}

\begin{remark}
For simplicity, the discussion in this section complexifies the surfaces in the flat infinite part of the interface geometry. In practice, the coordinates can be complexified using a smooth function and in the non-flat regions as well, as long as the parameterization is analytically available.
\end{remark}

We now show how the above quadrature corrections can be used to compute the action of these integral operators on smooth densities. Suppose that $\boldsymbol{X}: \mathbb{R}^{2} \to \pa \Omega_{C} \subset \mathbb{C}^{3}$ is a smooth parameterization of a complexified surface $\pa \Omega_{C}$. For $\bx = \bX(\bv') \in \pa \Omega_{C}$, we first consider the evaluation of integrals of the following form
\begin{equation}
\label{eq:int-g-main_old}
\int_{\pa \Omega_{C}} \frac{f(\by)}{r(\bx-\by)^{2s}} \, dS 
\end{equation}
where $f$ is any function with $f \circ S$ real analytic in a neighborhood of $\bv'$ and decaying sufficiently rapidly at infinity so that the integrand in the previous expression is absolutely integrable. The function $r(\bx)$ is the {\it complexified} distance function, as defined in~\cref{eq:rdef}.

Next, we define the vector-valued function $\bz = (z_{1},z_{2},z_{3}): \mathbb{R}^2 \to \mathbb{C}^3$ by
\begin{equation}
\bz = \partial_{v_{1}} \bX(v_{1},v_{2}) \times \partial_{v_{2}} \bX(v_{1},v_{2}) \, 
\end{equation}
and denote the Jacobian of the map $\bX$ by $J(\bv)$ given by
\begin{equation}
\label{eq:jac}
J(\bv) = \sqrt{z_{1}^2 (\bv) + z_{2}^2 (\bv) + z_{3}^2 (\bv)} \, ,
\end{equation}
In a slight abuse of notation, we set $f(\bv)  = f(\bX(\bv))$ and set $g(\bv) = f(\bv) J(\bv)$. 

Using these definitions, we obtain
\begin{equation}
\label{eq:int-g-main}
\int_{\pa \Omega_{C}} \frac{f(\by)}{r(\bx-\by)^{2s}} \, dS = \int_{\bR^2} \frac{f(\bv) }{r(\bX(\bv') - \bX(\bv))^{2s}} J(\bv) \, d \bv \,  = \int_{\bR^2} \frac{g(\bv)}{r(\bX(\bv') - \bX(\bv))^{2s}} \, d\bv,
\end{equation}

Without loss of generality, we take $\bv' = \bzero$, and $\bX(\bzero) = \bzero$. Discretizing the integral in~\cref{eq:int-g-main} using the punctured trapezoidal rule with spacing $h$ we obtain
\begin{equation}
\label{eq:int_withr}
I_{r} = \int_{\bR^2} \frac{g(\bv)}{r(\bX(\bv))^{2s}} \, d\bv \approx T_{r,h} = {\sum_{j \in \bZ^2}}' \frac{g(\bj h)}{r(\bX(\bj h))^{2s}}  h^2  \, .
\end{equation}
Similar to the error in discretization of~\cref{eqn:int} using the punctured trapezoidal rule, the error $|I_{r} - T_{r,h}|$ is also $O(h^{2-2s})$ owing to the singularity in the denominator. In the following theorem, applying the results of Theorem \ref{thm:diag}, we obtain a diagonal correction which improves the convergence rate to $O(h^{4-2s})$. Given the analytic extension of the Epstein zeta function to complex symmetric first fundamental forms, the proof is a straight-forward adaptation of the one in~\cite{wu2021corrected, wu2023unified}.

\begin{theorem}
\label{thm:diag2}
Let $\pa \Omega_{C}$ be a smooth parametric surface in $\mathbb{C}^{3}$ parameterized by the chart $\bX: \bR^2 \to S$, with $\bX(\bzero) = \bzero$. 
Then, 
\begin{equation}
|I_{r} - T_{r,h} - Z_{A}(s) g(\bzero) h^{2-2s}| = O(h^{4-2s}) \, ,
\label{eq:lap_diag_corr2}
\end{equation}
where $A$ is the first fundamental form of $\bX$ at $\bzero$.
\end{theorem}
\begin{proof}
Let $\delta(\bv) = r(\bX(\bv))^2-Q_{A}(\bv)$. We have $Q_{A}(\bv) = O(|\bv|^2)$ and $\delta(\bv) = O(|\bv|^3)$, so $\frac{\delta(\bv)}{Q(\bv)} = O(|\bv|)$. Expanding $r(\bX(\bv))^{-2s}$ with respect to $\frac{\delta(\bv)}{Q(\bv)}$ gives
\begin{equation}
\begin{aligned}
\frac{1}{r(\bX(\bv))^{2s}} &= \frac{1}{Q(\bv)^{s}}\left(1-s\frac{\delta(\bv)}{Q(\bv)}+\frac{s(s+1)}{2}\left(\frac{\delta(\bv)}{Q(\bv)}\right)^2+O(|\bv|^3)\right)\\
&= \frac{1}{Q(\bv)^s} - s\frac{\delta(\bv)}{Q(\bv)^{s+1}}+\frac{s(s+1)}{2}\frac{\delta(\bv)^2}{Q(\bv)^{s+2}} + O(|\bv|^{3-2s}) \, ,
\end{aligned}
\label{eq:binomial}
\end{equation}
provided $|\bv|$ is sufficiently small.

As in the proof of Theorem \ref{thm:diag}, let $\eta(\bv)$ be a radially symmetric, compactly supported $C^{\infty}$ function on $\bR^2$ with support in an interval $(-a,a)^2$, and $\eta(\bv) \equiv 1$ on $[-a/2,a/2]^2$. 
There exists $c_{\ell}$, $\ell=0,1,2,3$, such that
\begin{equation}
\label{eq:tayl2}
\begin{aligned}
g(\bv) \delta (\bv) &= \sum_{\ell=0}^{3} c_{\ell} v_{1}^{\ell} v_{2}^{3-\ell} + O(|\bv|^{4}) \, , \\
\end{aligned}
\end{equation}
holds for $\bv$ in the vicinity of $\bzero$. Note that it is always possible to do this since $\delta(\bv) = O(|\bv|^3)$ in the vicinity of the origin.
Again as  in the proof of Theorem \ref{thm:diag}, we first split the integrand into two parts
\begin{equation}
I_{r} = \int_{\bR^2} \frac{g(\bv) \eta(\bv)}{r(\bX(\bv))^{2s}} {\rm d}\bv + \int_{\bR^2} \frac{g(\bv)(1- \eta(\bv))}{r(\bX(\bv))^{2s}} {\rm d}\bv \, 
\end{equation}
and observe that
\begin{equation}
\label{eq:non-sing}
\left|\frac{g(\bv)(1- \eta(\bv))}{r(\bX(\bv))^{2s}} - {\sum_{\bj \in \bZ^2}}' \frac{g(\bj h)(1 - \eta(\bj h))}{r(\bX(\bv))^{2s}} \right| = O(h^{-M}) \, ,
\end{equation}
for all $M>0$. 

We now turn our attention to the singular term in the vicinity of the origin
\begin{equation}
\int_{\bR^2} \frac{g(\bv) \eta(\bv)}{r(\bX(\bv))^{2s}}{\rm d}\bv = \underbrace{\int_{\bR^2} g(\bv) \eta (\bv)\left(\frac{1}{r(\bX(\bv))^{2s}}-\frac{1}{Q_{A}(\bv)^{s}}\right){\rm d}\bv}_{I_1}+ \underbrace{\int_{\bR^2} \frac{g(\bv) \eta(\bv)}{Q_{A}(\bv)^{s}}{\rm d}\bv}_{I_2} \, .
\end{equation}
From Theorem \ref{thm:diag} it is clear that 
\begin{equation}
I_2 =  {\sum_{\bj \in \bZ^2}}' \frac{g(\bj h) \eta (\bj h))}{Q_{A}(\bj h)^{s}}h^2 - Z_{A}(s) g(\bzero) h^{2-2s} + O(h^{4-2s}).
\label{eq:I2}
\end{equation}
and hence it suffices to show that
\begin{equation}
I_1= {\sum_{\bj \in \bZ^2}}' g(\bj h) \eta(\bj h)\left(\frac{1}{r(\bX(\bj h))^{2s}}-\frac{1}{Q_{A}(\bj h)^{s}}\right)h^2+ O(h^{4-2s}).
\label{eq:I1}
\end{equation}

To show~\cref{eq:I1} holds, first note that~\cref{eq:binomial,eq:tayl2} imply that the integrand of $I_1$ for small $\bv$ is given by
\begin{equation}
\begin{aligned}
g(\bv)\eta(\bv)\left(\frac{1}{r(\bX(\bv))^{2s}}-\frac{1}{Q_{A}(\bv)^{s}}\right) &= - s\frac{g(\bv)\eta(\bv)\delta(\bv)}{Q_{A}(\bv)^{s+1}} + O(|\bv|^{2-2s})\\
&= -s\eta(\bv) \sum_{\ell=0}^{3} c_{\ell}\frac{v_{1}^{\ell} v_{2}^{3-\ell}}{Q_{A}(\bv)^{s+1}} + O(|\bv|^{2-2s}) \, .
\end{aligned}
\end{equation}
The leading term is an odd function, and hence both its integral and trapezoidal rule approximation vanish. The $O(|\bv|^{2-2s})$ remainder implies that the error of the trapezoidal rule when applied to $I_1$ is $O(h^{4-2s})$.

Thus, the theorem follows from the combination of~\cref{eq:non-sing,eq:I1,eq:I2}.
\end{proof}

We now return to our discussion of the layer potentials arising in the solution of the Helmholtz equation. For simplicity, we illustrate the computation of $\cS[\sigma](\bx)$ using the diagonally corrected quadrature rule discussed above. Analogous results hold for the other layer potentials. In light of the previous theorem, we find that 
\begin{equation}
\begin{aligned}
\cS[\sigma](\bX(\bl h)) = \sum_{\substack{\bj \in J_{N} \\ \bj \neq \bl}} &\frac{e^{ ik r(\bX(\bl h) - \bX(\bj h))}}{4 \pi r(\bX(\bl h) - \bX(\bj h))} J(\bj h) \sigma(\bX(\bj h))h^2 \\ 
&+\frac{\sigma(\bX(\bl h)) J(\bl h) h}{4\pi} \left(-Z_{A_{\bl}}(1/2) +ikh\right)+O(h^3)\, ,
\end{aligned}
\end{equation}
where $J(\bv)$ is the Jacobian of the parametrization $\bX$ defined in~\cref{eq:jac}, $J_{N}$ is the tensor product index set defined in~\cref{eq:tens-set}, and $N$ is chosen large enough so that $|\sigma(\bj h)| <\varepsilon$, for all $\bj \in \bZ^{2} \setminus J_{N}$, with $\varepsilon$ being a user defined tolerance. The last ``$ikh$'' term is a correction for the leading regular component that is omitted by the punctured trapezoidal rule.
\begin{remark}
The finite truncation of the index set to $J_{N}$ for numerical implementation results in an additional $O(\varepsilon)$ in the discretization of the layer potential. When using higher order corrections, we ignore the contribution of the quadrature corrections for points $\bj$ near the boundary of the index set $J_{N}$ which are outside of the set $J_{N}$. The additional error incurred in this setting is also $O(\varepsilon)$.
\end{remark}

\section{Numerical results}\label{sec:numerics}

In this section, we illustrate the efficacy of the corrected trapezoidal rule, particularly as applied to the evaluation of layer potentials on complexified surfaces, as described in Section \ref{sec:lay_pot}. For the imaginary part, we use a smooth mollifier defined in terms of the function $\phi$ given by
\begin{equation}
\phi(x)=-\frac12\int_x^\infty\mathrm{erfc}(t)\,{\rm d} t\equiv \frac{1}{2}\left(x\,\mathrm{erfc}(x) - \tfrac{e^{-x^2}}{\sqrt{\pi}}\right)
\end{equation}

\subsection{Numerical illustrations}
\paragraph{Example 1: Evaluation of layer potentials on complexified surfaces} We showcase the convergence of the modified trapezoidal rule on two surfaces. The first surface $\pa \Omega^{(1)}$ is a Gaussian bump parameterized by 
\begin{equation}
\pa \Omega^{(1)}=\{(x_{1},x_{2},x_{3})~|~x_{3}(x_{1},x_{2}) = -6\exp(-0.05(x_{1}^2+x_{2}^2)), -\infty\leq x_{1}, x_{2}\leq \infty\}
\end{equation} 
 Let $\bX^{(1)}$ denote the parametrization of the complexified version of $\pa \Omega^{(1)}$ given by
 \begin{equation}
 \begin{aligned}
 X^{(1)}_{1}(v_{1}, v_{2}) &= v_{1} + i \psi(v_{1}) \\
 X^{(1)}_{2}(v_{1}, v_{2}) &= v_{2} + i \psi(v_{2}) \\
 X^{(1)}_{3}(v_{1}, v_{2}) &= x_{3}(X^{(1)}(v_{1}), X^{(2)}(v_{2})) \, ,
 \end{aligned}
 \end{equation}
 where
 \begin{equation}
 \label{eq:psidef}
 \psi(v) = \phi\left(0.75(v+10)\right)-\phi\left(-0.75(v-10)\right) \, .
 \end{equation}
 For evaluating the layer potentials, the surface is truncated in the square $\bv \in [-30,30]^2$.
 The second surface $\pa \Omega^{(2)}$ is a slanted cylinder with radius $R=2.5$ and axis direction $(a_1,a_2,a_3)=(0.5,0.5,1)$. 
 Let $\bX^{(2)}$ denote the parametrization of $\pa \Omega^{(2)}$, given by
\begin{equation}
\begin{aligned}
X^{(2)}_{1}(v_{1},v_{2})&=R\cos v_{1}+a_1(v_{2} + i\psi(v_{2})),\\
X^{(2)}_{2}(v_{1},v_{2})&=R\sin v_{1} +a_2(v_{2} + i\psi(v_{2})),\\
X^{(2)}_{3}(v_{1},v_{2})&= a_3(v_{2} + i\psi(v_{2})).
\end{aligned}
\end{equation}
where $0\leq v_{1} < 2\pi$, $-\infty <v_{2}<\infty$, and $\psi$ is as defined in~\cref{eq:psidef}. As in $\pa \Omega^{(1)}$, the variable $v_{2}$ is truncated to $[-30,30]$.

We evaluate both the Helmholtz single and double layer potentials using the corrected trapezoidal rule with $3^\text{rd},$ $5^\text{th}$, and $7^\text{th}$ order corrections based on the implementation in \cite{wu2023unified}. We test the convergence on the complexified Gaussian bump $\pa \Omega^{(1)}$ at the target point with $(v_{1},v_{2}) = (-7.5, -9.375)$, and on the complexified cylinder $\pa \Omega^{(2)}$ at the target point with $(v_{1},v_{2}) = (1.2\pi, -8.6372)$, see~\cref{fig:conv_one_point}. The corresponding density $\sigma$ is chosen to be
$$
\begin{aligned}
\sigma(v_{1},v_{2})&=(\sin(0.6v_{1}+2) - 3\cos(0.7v_{2}-\pi))+i\exp(\sin(v_{1}+1)-\cos(0.2v_{2})) & &\text{on $\pa \Omega^{(1)}$}\,,\\
\sigma(v_{1},v_{2})&=(\sin(3v_{1}+2) - 3\cos(0.7v_{2}-\pi))+0.3i\,v_{1} \cos(2v_{2}) & &\text{on $\pa \Omega^{(2)}$} \,.
\end{aligned}
$$
The first fundamental forms at the target points are $E = 1 + 6 \times 10^{-3}i, F=1.39 \times 10^{-5}, G = 0.9638 + 0.3805i$ on $\pa \Omega^{(1)}$, and $E = 6.25, F = -0.2765 - 0.0461i, G = 1.4582 + 0.5006i$ on $\pa \Omega^{(2)}$, both of which require evaluation of the Epstein zeta function at a complex quadratic form.
We observe the expected orders of convergence for both layer potentials and on both surfaces. 

\begin{figure}[h!]
\centering
\includegraphics[width=\textwidth]{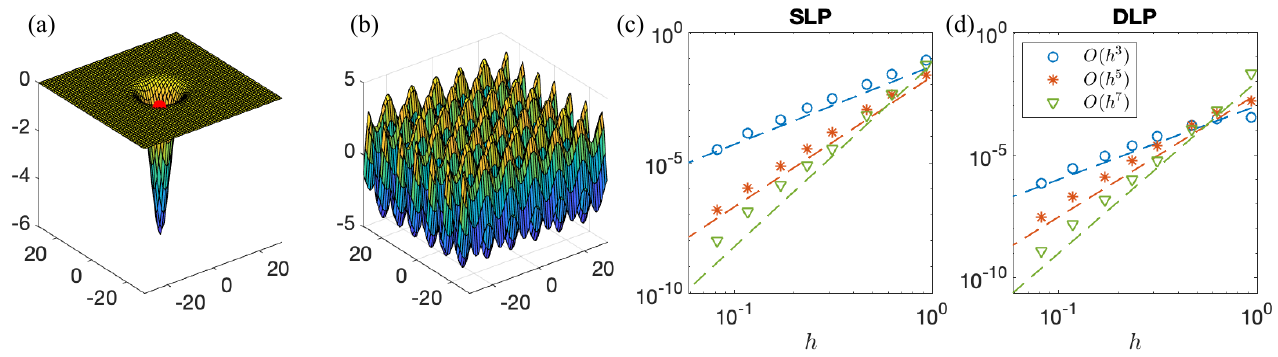}\\
\includegraphics[width=\textwidth]{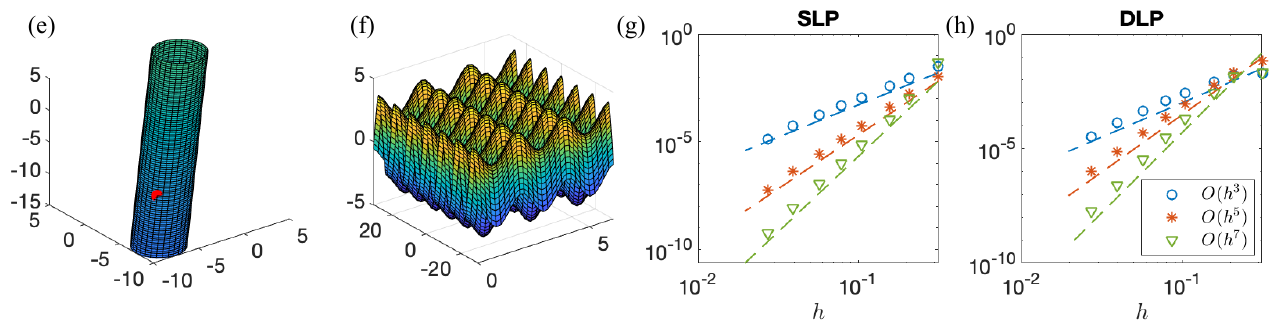}
\caption{Evaluation of Helmholtz single and doulbe layer potentials on complexified surfaces using corrected trapezoidal rules. \textbf{Top row} plots (a) real part of a Gaussian bump and target point (red dot), (b) real part of density $\sigma(v_{1},v_{2})$, and $O(h^3)$--$O(h^7)$ convergence of (c) single-layer and (d) double-layer potentials at the target. The dashed lines give reference to the corresponding exact orders of convergence. \textbf{Bottom row} (e)--(h) is analogous to the top row, but for a cylindrical surface.}
\label{fig:conv_one_point}
\end{figure}

\paragraph{Example: Dirichlet Helmholtz problem}
We conclude with an example, highlighting the use of the above quadrature method for solving Dirichlet Helmholtz problems in half-spaces. Specifically, using the complexified boundary integral equation described in~\cref{sec:lay_pot}, we solve the Dirichlet problem in the region lying above the surface in~\cref{fig:solve} (upper left), using the complexification of the $x_{1}$ and $x_{2}$ variables shown in the upper right panel of~\cref{fig:solve}.

More concretely, the real part of the surface is parameterized by
$$\bX^{(r)}(\bv) = \left(v_{1},v_{2},e^{-(v_1^2+v_2^2)/8}\left[\cos(1.9 v_1+0.95 v_2)+\sin(v_1+1.55 v_2) \right] \right)$$
and the imaginary part is parameterized by
$$\bX^{(i)}(\bv) = \frac{1}{2}\left(\psi(v_1),\psi(v_2),0 \right) \, ,$$
and $\psi$ is given by~\cref{eq:psidef}.
The complexified surface is then parameterized by
$$\bX(\bv) = \bX^{(r)}(\bv) + i \bX^{(i)}(\bv),$$
which we discretize on the range $-20 \le v_1,v_2 \le 20.$ 

In order to have an analytic solution for comparison, the data on the half-space was chosen to be the field generated by a point source at $\boldsymbol{s} = (1,1,-2)$ (below the surface); in which case, the solution above the plane is simply the fundamental solution due to a point source at $\boldsymbol{s}.$

The integral operator in~\cref{eq:comb-rep} was discretized using $5^{\rm th}$ order corrections with $h = 1/16$ and $640 \times 640$ points. Due to the relatively large size of the problem, in order to accelerate the computation, a minor modification of~\cite{minden2017,sushnikova2023} was used to form a compressed representation of the inverse of the discretization of the operator in~\cref{eq:comb-rep}. Plots of the resulting solution to the boundary integral equation, and the error in the corresponding solution to the Helmholtz Dirichlet problem, are shown in~\cref{fig:solve}. We point out that in the bottom right panel of \cref{fig:solve}, the exponential growth near the edges is expected --- accuracy in the solution is only guaranteed in or near the region $\{\bx\in\mathbb{R}^3 ~|~ |x_1|<10,\,|x_2|<10\},$ i.e. where the imaginary part of the complexified surface is close to zero.

\begin{figure}[h!]
\centering
\includegraphics[width=0.45\textwidth]{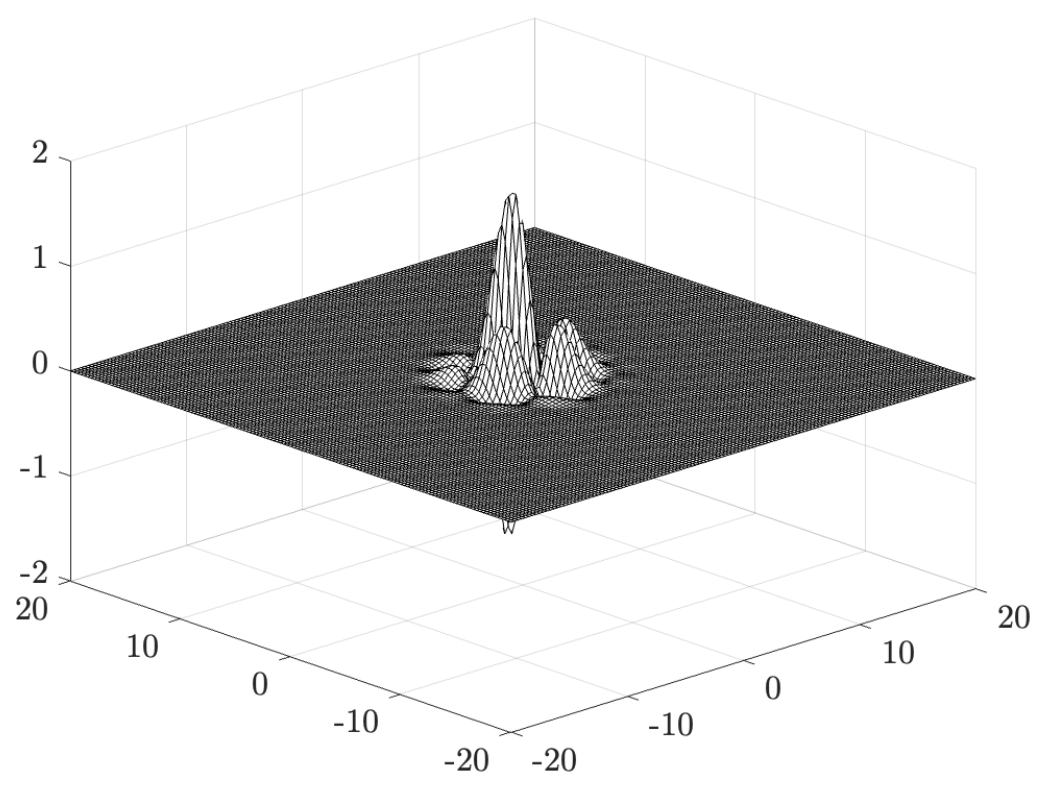}
\includegraphics[width=0.45\textwidth]{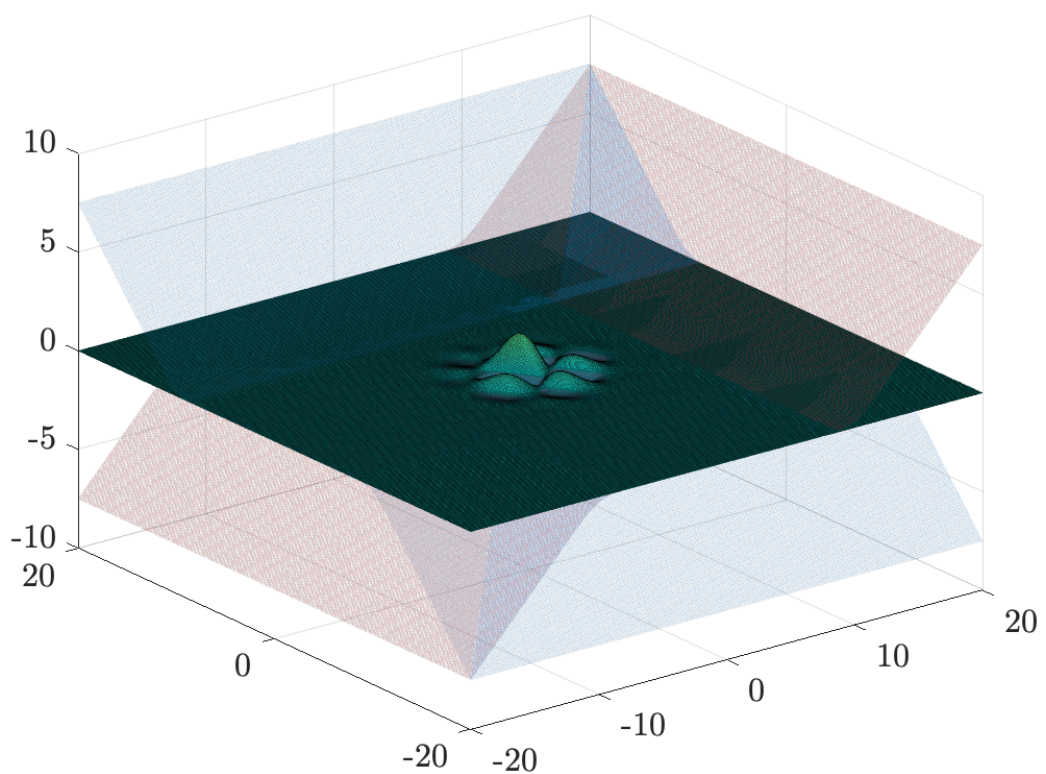}
\includegraphics[width=0.45\textwidth]{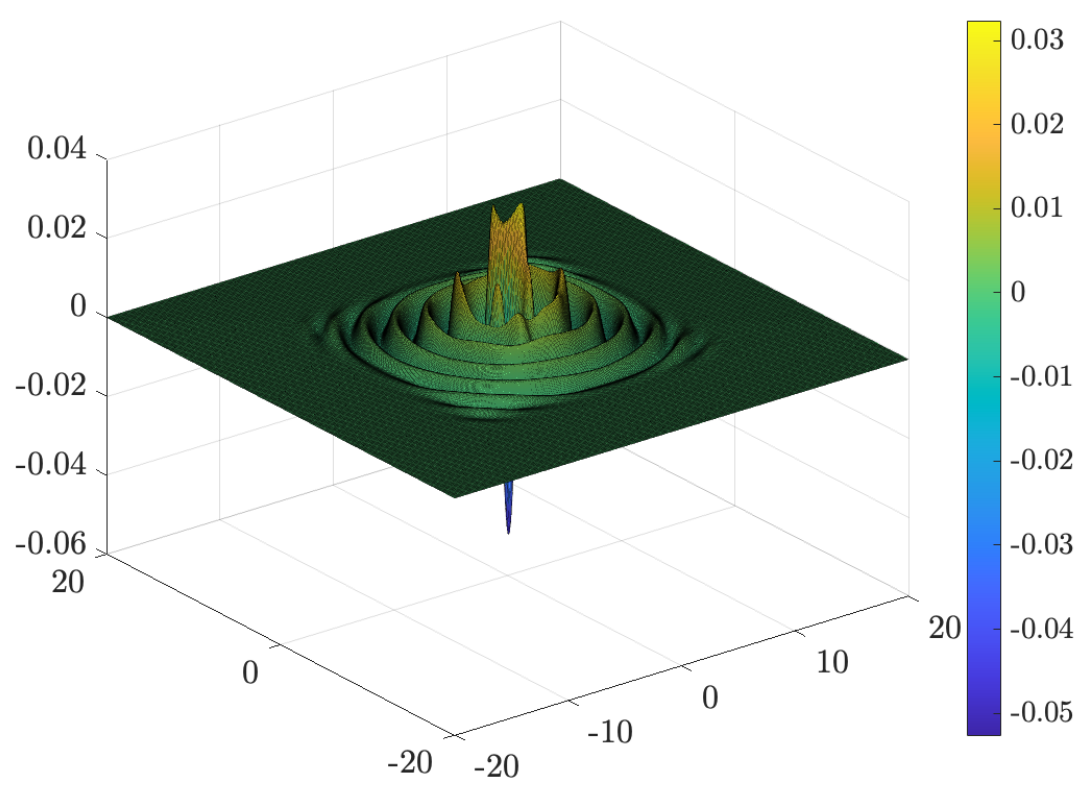}
\includegraphics[width=0.45\textwidth]{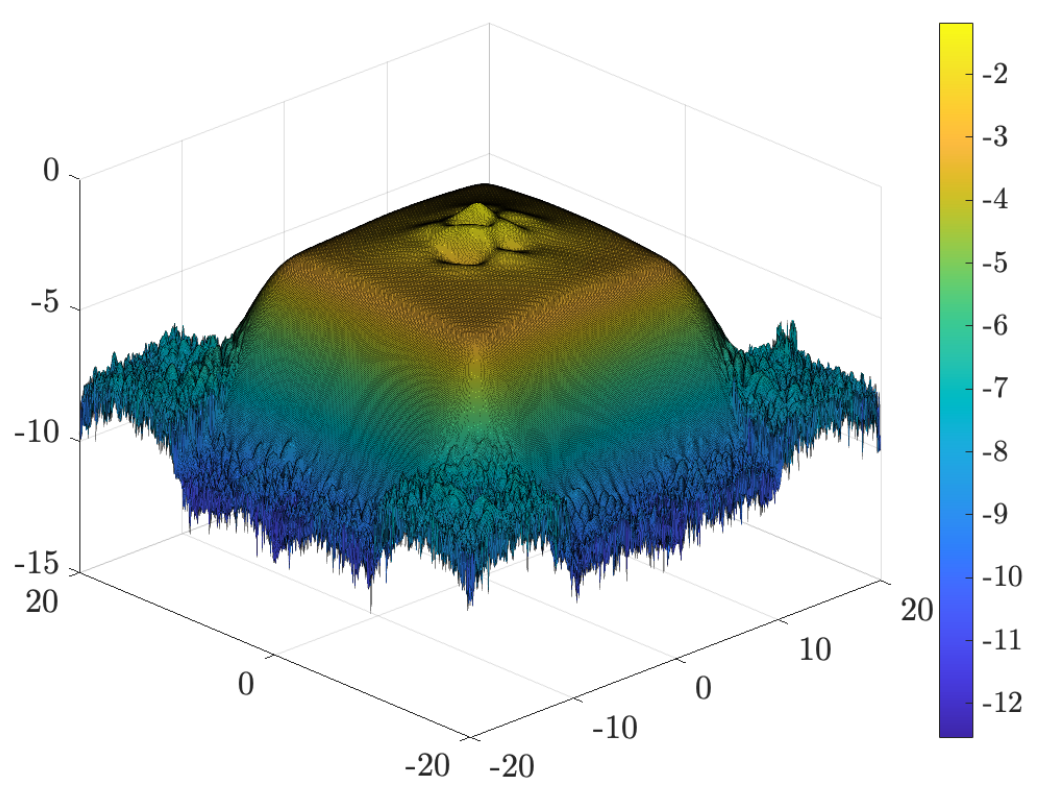}\\
\includegraphics[width=0.45\textwidth]{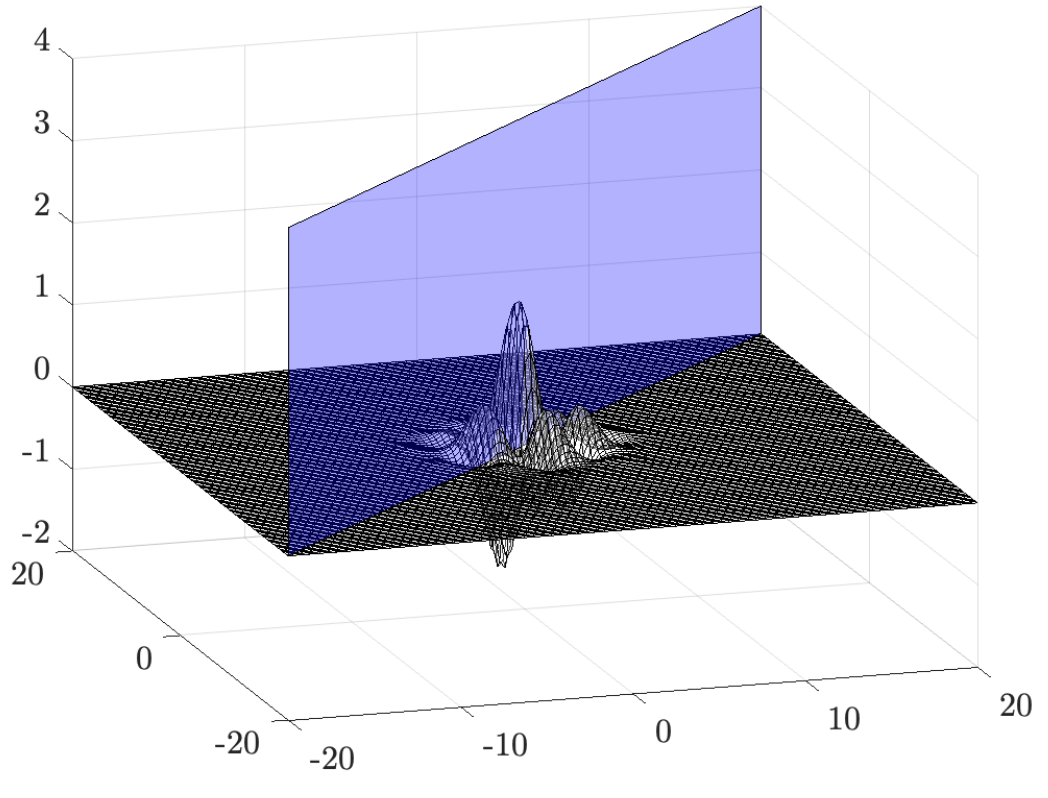}
\includegraphics[width=0.45\textwidth]{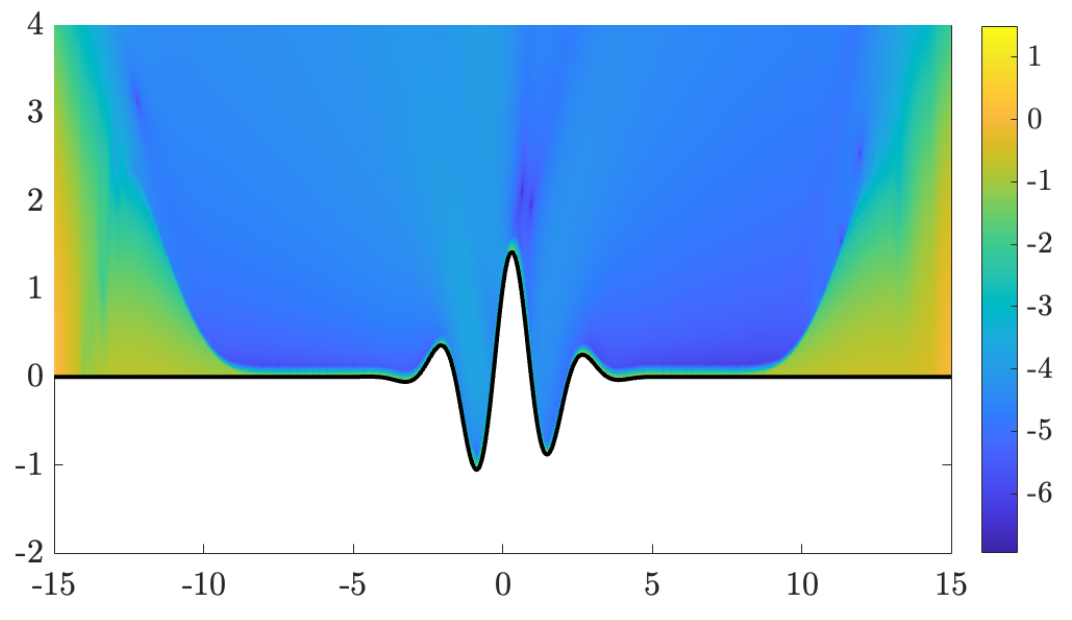}
\caption{Solution of a Dirichlet Helmholtz problem with wavenumber $k=3$ in a half-space. The solution was obtained using a combined field boundary integral formulation on a complexified surface. The data was generated by a point source at $(1,1,-2)$ (below the surface) and the solution computed in the region above the surface. $5^{\rm th}$ order corrections were used. {\bf Top left:} real part of the surface. {\bf Top right:} red,blue  - $z$ coordinate is the imaginary part of the $x$ and $y$ components, respectively. {\bf Middle left:} real part of the density (the solution of the boundary integral equation). {\bf Middle right:} $\log_{10}$ plot of the absolute value of the density. {\bf Bottom left:} a plot of the cross section taken. {\bf Bottom right:} error plot ($\log_{10}$ relative error) on a cross-section.}
\label{fig:solve}
\end{figure}

%\clearpage

\section{Conclusion}

We have described an extension of the classical corrected trapezoidal rule to a class of weakly singular and singular integrals on complex domains in $\mathbb{C}^2.$ The method is based on properties of the Epstein zeta function and Wigner limits. The proofs, and the resulting method, rely on the analysis of a novel extension of the Epstein zeta function to complex symmetric quadratic forms. The performance of the resulting quadrature was demonstrated on several examples for which $3$rd, $5$th, and $7$th order quadrature corrections were constructed. Moreover, we illustrated the application of our quadrature method to solving scattering problems involving infinite geometries, which are reduced to numerically finite problems when complexified.

We briefly mention several possible directions for future research. Firstly, the methods described here should extend quite naturally to higher dimensions, where they could be useful for solving certain volume integral equations in infinite domains. Secondly, though the infinite portion of many geometries can be discretized easily via a standard equispaced mesh, for many examples triangular meshes are more convenient. This is particularly true for surfaces with regions of high curvature. A natural direction is then to formulate fast and accurate quadrature rules for triangular patches on complex surfaces. Alternatively, an attractive option would be to use triangular patches for the parts of the surface which are purely real, and knit this together with tensor-product trapezoidal discretizations of the complex portions.

\section{Acknowledgements}

The authors would like to thank Charles Epstein, Leslie Greengard, Shidong Jiang, and Per-Gunnar Martinsson for many useful discussions.

\bibliographystyle{plain}
\bibliography{bib}

\appendix

\section{Proof of~\Cref{prop:zeta_ext}}\label{app:split}
The proof is essentially identical to the one presented in~\cite{crandall1998}, wherein the case of real positive definite matrices $A$ was considered. We include it here for completeness. We begin by recalling the following identity
$$\frac{1}{Q_A(\bv)^s} = \frac{\pi^s}{\Gamma(s)}\int_0^\infty t^{s-1} e^{-\pi Q_A(\bv) t}\,{\rm d} t,$$
which holds provided that $\Re{(Q_A(\bv))}>0$ and $\Re{(s)} >0.$ Additionally, we require the following identity which is a simple consequence of the Poisson summation formula \cite{poisson1821}. 

\begin{lemma}
$$ \sum_{\bj \in \bZ^2} e^{-\pi t Q_A(\bj)} = \frac{1}{t}\frac{1}{\sqrt{{\rm det}\,A}}  \sum_{\bj \in \bZ^2}  e^{-\pi Q_{A^{-1}}(\bj)/t}$$
for $\Re{(Q_A(\bv))}\geq 0$, with equality only holding for $\bv = \bzero$,  and $t>0.$
\end{lemma}

\begin{proof}
    We begin by observing that the function $F$ defined by
$$F(\bx) = \sum_{\bj \in \bZ^{2}} e^{-Q_{A}(\bx + \bj)}$$
is doubly periodic with period 1. In particular,
\begin{equation}
\label{eq:fdef}
F(\bx) = \sum_{\bk\in \bZ^2} e^{2\pi i \bk \cdot \bx} \hat{F}_{\bk},
\end{equation}
where
$$\hat{F}_{\bk} = \int_{[0,1]^2} e^{-2\pi i \bk \cdot \by} F(\by)\,{\rm d}\by.$$
Note, that the Fourier coefficients can be further simplified as follows
\begin{equation}
\begin{aligned}
\hat{F}_{\bk} &= \int_{[0,1]^2} e^{-2\pi i \bk \cdot \by} \sum_{\bj \in \bZ^2} e^{-Q_{A}(\by + \bj)}\,{\rm d}\by \, ,\\
&= \int_{\bR^2} e^{-2\pi i \bk \cdot \by} e^{-Q_{A}(\by)}\,{\rm d}\by \, \\
&= \frac{\pi}{\sqrt{{\rm det}\,A}} e^{-\pi^2 Q_{A^{-1}}(\bk)} \, .
\end{aligned}
\end{equation}
The result then follows from evaluating~\cref{eq:fdef} at $\bx = 0$, and a simple rescaling of the matrix $A$.
\end{proof}
We return back to the proof of~\Cref{prop:zeta_ext}.
After summing the first identity over $\bj \in \bZ^{2} \setminus \{\bzero\}$, one finds that for $\Re{(s)} >1,$
$$ \sum_{\bj \in \bZ^2}\,' \frac{1}{Q_A(\bj)^s} = \frac{\pi^s}{\Gamma(s)} \sum_{\bj \in \bZ^2}\,'\int_0^\infty t^{s-1} e^{-\pi Q_A(\bj) t}\,{\rm d}t. $$
We split the domain of integration into two parts: $[0,1)$ and $[1,\infty),$ which we denote by $I_1$ and $I_2$ respectively. In the first integral, $I_1,$ after changing variables by setting $t=1/r$ we obtain
\begin{align*}
\frac{\Gamma(s)}{\pi^s}I_1&=\int_0^1  t^{s-1} \sum_{\bj \in \bZ^2}\,' e^{-\pi Q_A(\bj)t}\,{\rm d}t\\
&=\int_0^1  t^{s-1} \sum_{\bj \in \bZ^2} e^{-\pi Q_A(\bj)t}\,{\rm d}t - \int_0^1 t^{s-1} \,{\rm d}t\\
&=\int_0^1  t^{s-2} \frac{1}{\sqrt{{\rm det} A}} \sum_{\bj\in \bZ^2} e^{-\pi Q_A^{-1}(\bj)/t}\,{\rm d}t -\frac{1}{s }  \\
&=\int_0^1  t^{s-2} \frac{1}{\sqrt{{\rm det}A}} \sum_{\bj \in \bZ^2}\,'\, e^{-\pi Q_{A^{-1}}(\bj)/t}\,{\rm d}t+\frac{1}{\sqrt{{\rm det}A}}\int_0^1 t^{s-2}\,{\rm d}t  -\frac{1}{s }  \\
&=\int_1^\infty  r^{(1-s)-1} \frac{1}{\sqrt{{\rm det}A}} \sum_{\bj \in \bZ^2}\,'\, e^{-\pi  r Q_{A^{-1}}(\bj)}\,{\rm d}r+\frac{1}{(s-1)\sqrt{{\rm det}A}} -\frac{1}{s }  \\
&= \frac{1}{\sqrt{{\rm det}A}}\sum_{\bj \in \bZ^2}\,'\,\frac{\Gamma(1-s,\pi Q_{A^{-1}}(\bj))}{\left(\pi Q_{A^{-1}}(\bj)\right)^{1-s}}+\frac{1}{(s-1)\sqrt{{\rm det}A}} -\frac{1}{s}.
\end{align*}
Similarly,
\begin{align*}
\frac{\Gamma(s)}{\pi^s}I_2&=\int_1^\infty t^{s-1} \sum_{\bj \in \bZ^2}\,' e^{-\pi Q_A(\bj)t}\,{\rm d}t= \sum_{\bj \in \bZ^2}\,'\,\frac{\Gamma(s,\pi Q_A(\bj))}{\left(\pi Q_A(\bj)\right)^s}.
\end{align*}
Upon putting these last two expressions together, we arrive at
\begin{align*}
\sum_{\bj \in \bZ^2}\,' \frac{1}{Q_A(\bj)^s} = &\frac{\pi^s}{\Gamma(s)}\Bigg\{\sum_{\bj}\,'\,\left[\frac{\Gamma(s,\pi Q_A(\bj))}{\left(\pi Q_A(\bj)\right)^s}+\frac{1}{\sqrt{{\rm det}A}}\frac{\Gamma(1-s,\pi Q_{A^{-1}}(\bj))}{\left(\pi Q_{A^{-1}}(\bj)\right)^{1-s}}\right]\\
&+\frac{1}{(s-1)\sqrt{{\rm det}A}} -\frac{1}{s}\Bigg\}.
\end{align*}
This identity holds for $\Re{(s)}>1$, and the expression on the right-hand side is analytic for all $s\in\mathbb{C}\backslash\{1\}$ (note that $\frac{1}{\Gamma(s)}$ is an entire function, which is analytic at $s=0$ in particular). Therefore the lattice sum is analytically continued by the above expression,
which completes the proof.

\section{Proof of~\Cref{prop:borwein}}\label{app:bor}
This appendix extends the Theorem 1 in~\cite{borwein89} to the complex symmetric quadratic form $Q_{A}$. The proof involves a minor modification of that in~\cite{borwein89}, which we include here for completeness. We first comment on the definition of $Q_{A}(\bv)^{-s}$. Since, we have assumed that $\mathrm{Re}(A)>0$, the numerical range of the quadratic form $Q_{A}(\bv)$ is contained in the right half of the complex plane. Hence, 
$Q_{A}(\bv)^{-s} = e^{-s\Log{(Q_{A}(\bv))}}$ can be defined via the principal branch of the logarithm with a branch cut on the negative real axis, i.e. $\Log{(z)} = \log{|z|} + i \textrm{Arg}(z)$, with $\textrm{Arg}(z) \in (-\pi, \pi]$.
\begin{theorem}
\label{thm:wigner}
Suppose that $W_{A}^{(N)}$, $Z_{A}^{(N)}$, and $I_{A}^{(N)}$ are as defined in~\cref{eq:wzi-def}. Then,
for $\Re{(s)} \in (0,1)$ , $W_{A}(s) = \lim_{N\to\infty} W_{A}^{(N)}(s)$ exists and coincides with the analytic extension of 
$Z_{A}(s) = \lim_{N\to \infty} Z_{A}^{(N)}(s)$ on that region.
\end{theorem}

\begin{proof}
We proceed similar to the proof in~\cite{borwein89} for the first part, which shows the existence of $\lim_{N \to \infty} (W_{A}^{(N)} - W_{A}^{(0)})$ via showing the absolute convergence of $\sum_{j=1}^{N-1} \delta_{j}(s) = \sum_{j=1}^{N} W_{A}^{(j)}(s) - W_{A}^{(j-1)}(s) = W_{A}^{(N)}(s) - W_{A}^{(0)}(s)$ as $N\to \infty$, when $s \in \Omega =\{s\, : \Re{(s)}>\varepsilon, \textrm{ and } |s|<R \}$. Observe that
\begin{equation}
\delta_{N}(s) = \sum_{\substack{\bj \in \bZ^2 \\|\bj|_{\infty} = N}} \int_{[-\half,\half]^2} (Q_{A}(\bj)^{-s} - Q(\bj + \bx)^{-s}) \, {\rm d}\bx \, .
\end{equation}
Letting $f(x,y) = Q_{A}(\bj + \bx)^{-s}$ with $|\bj|_{\infty} = N$, and $\bx \in [-1/2,1/2]^2$, we will first show that
\begin{equation}
\label{eq:f-tayl-est}
|f(\bx) - f(\bzero) - x_{1} \partial_{x_{1}}f (\bzero) - x_{2} \partial_{x_{2}} f (\bzero)|  \leq M N^{-\varepsilon-2} \, ,
\end{equation}
where $M$ is a constant independent of $s \in \Omega$ and $N$. 
This follows from a Taylor expansion of $f$, and noting that
\begin{equation}
\partial_{x_{1}x_{1}}f(\bx) = s \left( s + 1 \right) Q_{A}(\bj + \bx)^{-s-2} \left( 2E(j_{1}+x_{1}) + 2F(j_{2}+x_{2}) \right)^2 - 4EsQ(\bj + \bx)^{-s-1} \, .
\end{equation}
Since $\Re{(A)} > 0$, we observe that for $s\in \Omega$, $\ell>0$, 
\begin{equation}
\label{eq:qest}
\begin{aligned}
\sup_{\substack{|\bj|_{\infty} = N\\ \bx\in [-1/2,1/2]^2} } |Q_{A}(\bj + \bx)|^{-s - \ell} &\leq |\Re(Q_{A}(\bj + \bx))|^{-\Rez{s} - \ell} \leq C_{A} N^{-2\varepsilon - 2\ell} \, .
\end{aligned}
\end{equation}
Combining the estimate with the definition of $\partial_{x_{1}x_{1}}f$, we conclude that there exists an $M_{1}$ independent of $s\in \Omega$ such that
\begin{equation}
|\partial_{x_{1}x_{1}}f| \leq M_{1} N^{-2\varepsilon -2} \, .
\end{equation}
Similar results hold for $\partial_{x_{1}x_{2}}f$, and $\partial_{x_{2}x_{2}}f$. Combining these estimates on the second derivatives, we conclude that~\cref{eq:f-tayl-est} holds.

Note that~\cref{eq:qest}, also shows that $Z_{A}(s) := \lim_{N\to \infty} Z_{A}^{(N)}(s)$ exists as long as $\Re{(s)}>1$, and is an analytic function. This follows from the absolute summability of $Q_{A}(\bj)^{-s}$ over $\bZ^2$ as long as $\Re{(s)}>1+\varepsilon$, for any $\varepsilon>0$.

Returning back to the definition $\delta_{N}$, note that
\begin{equation}
\begin{aligned}
|\delta_N (s)| &= \left| \sum_{\substack{\bj \in \bZ^2 \\ |\bj|_{\infty} = N}} \int_{[-\half,\half]^2}  (Q_{A}(\bj)^{-s} - Q_{A}(\bj + \bx)^{-s}) \, {\rm d}\bx   \right|  \\
&= \left| \sum_{\substack{\bj \in \bZ^2 \\ |\bj|_{\infty} = N}} \int_{[-\half,\half]^2} (f(\bzero)-f(\bx)) \, {\rm d}\bx  \right|
 \\
&= \left| \sum_{\substack{\bj \in \bZ^2 \\ |\bj|_{\infty} = N}} \int_{[-\half,\half]^2} (f(\bzero) + x_{1} \partial_{x_{1}} f(\bzero) +x_{2} \partial_{x_{2}} f(\bzero)-f(\bx)) \, {\rm d}\bx  \right| \, \\
& \leq \sum_{\substack{\bj \in \bZ^2 \\ |\bj|_{\infty} = N}} M N^{-2\varepsilon - 2} \leq 8 M N^{-2\varepsilon - 1} \, . 
\end{aligned}
\end{equation}
where the third equality follows from the fact that \
\begin{equation*} 
\int_{[-\half,\half]^2} x_{1} \,{\rm d}\bx = \int_{[-\half,\half]^2} x_{2} \,{\rm d}\bx = 0 \, .
\end{equation*}

Since $\delta_{N}(s)$ are entire functions of $s$, it follows from the Weierstrass $M$ test that 
$\delta(s) := \sum_{j=1}^{\infty} \delta_{j}(s)$ is analytic for $s\in \Omega$, and hence in the half plane $\Re{(s)}>0$. 
This also implies that $W_{A}(s) := \lim_{N\to \infty} W_{A}^{(N)}(s) = \delta(s) + W_{A}^{(0)}(s)$ exists and is an analytic function for $\Re{(s)} \in (0,1)$. The restriction on the domain of analyticity arises from the fact that $W_{A}^{(0)}(s) = -I_{A}^{(0)}(s)$, which is well defined and analytic as along as $\Re{(s)} \in (0,2)$, owing to the singularity in the integrand near the origin.

The next part of the proof in~\cite{borwein89}, relies on identifying two functions which are equal to the negative of each other, and where one of them is analytic and has a closed form representation around $s=0$, and the other function
is analytic and has a closed form representation around $s=\infty$. In our case, in order to construct such a function, we first observe that
\begin{equation}
\label{eq:qform-divform}
Q_{A}(\bx)^{-s} = \frac{1}{2(1-s)}\left( \partial_{x_{1}}(x_{1} Q_{A}(\bx)^{-s})  + \partial_{x_{2}}(x_{2} Q_{A}(\bx)^{-s}) \right) \, ,
\end{equation}
for any $s\in \mathbb{C} \setminus \{ 1 \}$. Let $\Omega\in \mathbb{R}^{2}$ denote the set such that $|Q_{A}(\bx)| >1$ for
$\bx \in \mathbb{R}^{2} \setminus \Omega$. Note that $\Omega$ is compact, contained in some disk centered at the origin, and contains the origin. Constructing such a set is always possible since $\Re{(A)}>0$. Finally, let
\begin{equation}
I_{1}(s) = \int_{\Omega} Q_{A}(\bx)^{-s} \, {\rm d}\bx \, ,\quad \textrm{and} \quad 
I_{2}(s) = \int_{\mathbb{R}^{2}\setminus \Omega} Q_{A}(\bx)^{-s} \, \, {\rm d}\bx \, .
\end{equation}
Owing to the singularity of $Q_{A}(\bx)^{-s}$ as a function of $\bx$ near $\bzero \in \Omega$, we note that the integral
$I_{1}$ converges for $\Re{(s)}<1$. And owing to the behavior of $Q_{A}(\bx)^{-s}$ at $\infty$, the integral $I_{2}$ converges for $\Re{(s)}>1$. However, using~\cref{eq:qform-divform}, and Green's theorem, we observe that
\begin{equation}
\begin{aligned}
I_{1}(s) &= \frac{1}{2(1-s)}\int_{\pa \Omega} Q_{A}(\bx)^{-s} (x_{1}\,{\rm d}x_{2} - x_{2}\,{\rm d}x_{1}) \, , \quad \textrm{and} \,, \\
I_{2}(s) &= -\frac{1}{2(1-s)}\int_{\pa \Omega} Q_{A}(\bx)^{-s} (x_{1}\,{\rm d}x_{2} - x_{2}\,{\rm d}x_{1}) \, .
\end{aligned}
\end{equation}
These integral forms for $I_{1}(s)$, and $I_{2}(s)$ are valid for all $s\neq 1$ in the right half plane, since $|Q_{A}(\bx)| = 1$ on $\pa \Omega$, and since $\Omega$ is contained in a ball with $|\pa \Omega| < \infty$. 
From this, we conclude that $I_{1}(s)$ and $I_{2}(s)$ extend analytically to the the whole $\mathbb{C}$ plane (except for $s=1$) where $I_{1} = -I_{2}$.

Recall that $W_{A}(s) + I_{1}(s) = \delta(s) - I_{A}^{(0)}(s) + I_{1}(s)$ is analytic for $\Re{(s)}>0$.
On the other hand, 
\begin{equation}
\lim_{N\to\infty} (I_{A}^{(N)}(s) - I_{1}(s)) = \lim_{N\to \infty} \int_{[-N-1/2, N+1/2]^2 \cap (\mathbb{R}^{2} \setminus \Omega)} Q_{A}(\bx)^{-s} \, {\rm d}\bx = I_{2}(s) \, ,
\end{equation}
which holds for $\Re(s)>1$ (and extends analytically for all $s\neq 1$). Thus for $\Re(s)>1$,
\begin{equation}
\begin{aligned}
W_{A}(s) +  I_{1}(s) &= \lim_{N\to \infty} (Z_A^{(N)}(s) - I_A^{(N)}(s)) + I_1(s)\\
&= \lim_{N\to \infty} Z_A^{(N)}(s) - (I_A^{(N)}(s) - I_1(s))\\
&= \lim_{N\to \infty} Z_{A}^{(N)}(s) - I_2(s)  \\
&= Z_{A}(s) + I_1(s) \, .
\end{aligned}
\end{equation}
This implies that $W_{A}(s)$ is an analytic extension of $Z_A(s)$ to the region $\Re{(s)}\in (0,1)$.
\end{proof}
\end{document}